\documentclass[a4paper,12pt,reqno]{amsart}
\usepackage{amsmath, amsfonts, amssymb, amsthm, mathrsfs, mathtools, mathalfa}
\usepackage{comment}
\usepackage{yfonts}
\usepackage{enumerate}
\usepackage{enumitem}
\usepackage{stmaryrd}
\usepackage{fancyhdr}
\usepackage{bm}
\usepackage[utf8]{inputenc}
\usepackage[pdfencoding=auto, hyperfootnotes=false]{hyperref}
\usepackage{tikz,tikz-cd}
\usepackage[hang,flushmargin]{footmisc}
\usepackage{graphicx}
\usepackage{xr}
\usepackage{extarrows}
\usepackage[dvistyle, colorinlistoftodos]{todonotes}
\usepackage[misc,geometry]{ifsym}

\usepackage{tikz}
\usetikzlibrary{matrix}
\tikzset{stretch/.initial=1}
\usetikzlibrary{calc}

\newcommand\drawloop[4][]%
   {\draw[shorten <=0pt, shorten >=0pt,#1]
      ($(#2)!\pgfkeysvalueof{/tikz/stretch}!(#2.#3)$)
      let \p1=($(#2.center)!\pgfkeysvalueof{/tikz/stretch}!(#2.north)-(#2)$),
          \n1= {veclen(\x1,\y1)*sin(0.5*(#4-#3))/sin(0.5*(180-#4+#3))}
      in arc [start angle={#3-90}, end angle={#4+90}, radius=\n1]%
   }

\makeatletter
\def\@tocline#1#2#3#4#5#6#7{\relax
  \ifnum #1>\c@tocdepth 
  \else
    \par \addpenalty\@secpenalty\addvspace{#2}%
    \begingroup \hyphenpenalty\@M
    \@ifempty{#4}{%
      \@tempdima\csname r@tocindent\number#1\endcsname\relax
    }{%
      \@tempdima#4\relax
    }%
    \parindent\z@ \leftskip#3\relax \advance\leftskip\@tempdima\relax
    \rightskip\@pnumwidth plus4em \parfillskip-\@pnumwidth
    #5\leavevmode\hskip-\@tempdima
      \ifcase #1
       \or\or \hskip 1em \or \hskip 2em \else \hskip 3em \fi%
      #6\nobreak\relax
    \dotfill\hbox to\@pnumwidth{\@tocpagenum{#7}}\par
    \nobreak
    \endgroup
  \fi}
\makeatother

\newtheorem{theorem}{Theorem}[section]
\newtheorem{lemma}[theorem]{Lemma}

\newtheorem{corollary}[theorem]{Corollary}
\newtheorem{proposition}[theorem]{Proposition}
\newtheorem{question}[theorem]{Question}

\newtheorem{conjecture}[theorem]{Conjecture}
\theoremstyle{definition}
\newtheorem{defn}[theorem]{Definition}
\newtheorem{remark}[theorem]{Remark}

\newcommand{\mc}{\mathcal}

\newcommand{\mf}{\mathbf}
\newcommand{\mb}{\mathbb}

\DeclareMathOperator{\ab}{Z}

\DeclareMathOperator{\tran}{\Theta}

\DeclareMathOperator{\q}{c}
\DeclareMathOperator{\ns}{X}
\DeclareMathOperator{\nss}{Y}

\DeclareMathOperator{\co}{\circ\hspace{-0.02 cm}}
\DeclareMathOperator{\cu}{C}

\DeclareMathOperator{\orb}{orb}

\DeclareMathOperator{\abph}{\mc{U}}

\newcommand*{\db}[1]{\llbracket #1\rrbracket}

\begin{document}
\vspace*{-1cm}

\title[On measure-preserving $\mb{F}_p^\omega$-systems of order $k$]{On measure-preserving $\mb{F}_p^\omega$-systems of order $k$}

\author{Pablo Candela}
\address{Universidad Aut\'onoma de Madrid and ICMAT\\ 
Ciudad Universitaria de Cantoblanco\\ Madrid 28049\\ Spain}
\email{pablo.candela@uam.es}

\author{Diego Gonz\'alez-S\'anchez}
\address{MTA Alfr\'ed R\'enyi Institute of Mathematics, Re\'altanoda u. 13-15.\\
Budapest, Hungary, H-1053}
\email{diegogs@renyi.hu}

\author{Bal\'azs Szegedy}
\address{MTA Alfr\'ed R\'enyi Institute of Mathematics, Re\'altanoda u. 13-15.\\
Budapest, Hungary, H-1053}
\email{szegedyb@gmail.com}

\begin{abstract}
Building on previous work in the nilspace-theoretic approach to the study of Host--Kra factors of measure-preserving systems, we prove that every ergodic $\mb{F}_p^\omega$-system of order $k$ is a factor of an Abramov $\mb{F}_p^\omega$-system of order $k$. This answers a question of Jamneshan, Shalom and Tao.
\end{abstract}
\maketitle

\vspace*{-1cm}
\section{Introduction}
The uniformity seminorms introduced by Host and Kra in ergodic theory \cite{HK-non-conv} are tools of great use for the analysis of various types of ergodic averages, with many combinatorial and number-theoretic applications (see for instance \cite{Nikos} or \cite{HKbook}). The \emph{Host--Kra factors} of an ergodic measure-preserving system, i.e.\ the system's characteristic factors for the Host--Kra seminorms, play a central role in this analysis\footnote{Such averages can also be analyzed via the characteristic factors used by Ziegler in \cite{Ziegler}. We distinguish Host--Kra factors from Ziegler factors since, as explained in \cite[\S 2.6]{ABS}, in general these two notions differ, even though in some contexts the notions may coincide (e.g.\ for $\mb{Z}$-actions).}. The study of these factors is therefore a major topic in this direction. Initiated in \cite{HK-non-conv} for $\mb{Z}$-actions, this study has been extended to other group actions in several works, including  \cite{ABB,BTZ,BTZ2,CScouplings,GL,Shalom,Shalom2,Shalom3}. 

This topic is also closely related to the analysis of Gowers norms\footnote{The Host--Kra seminorms are ergodic-theoretic analogues of the Gowers norms from \cite{GSz}.}  in arithmetic combinatorics, and especially to the topic of inverse theorems for these norms (in fact the two topics can be studied in a common framework, see \cite{CScouplings}). One of the first cases in which this relation was established and utilized is the case of actions of the additive group of the vector space $\mb{F}_p^\omega:=\bigoplus_{n=1}^\infty \mb{F}_p$ of prime characteristic $p$. This case attracted interest especially due to its connections with the setting of \emph{finite-field models} in arithmetic combinatorics \cite{Wolfsurvey}. In particular, in \cite{TZ-High} Tao and Ziegler established a variant of Furstenberg's correspondence principle which made it possible to prove the inverse theorem for the Gowers $U^{k+1}$-norm on vector spaces $\mb{F}_p^n$ (in the \emph{high-characteristic case}, i.e.\ for  $p>k$) by proving a result describing the structure of the Host--Kra factor of order $k$ for ergodic $\mb{F}_p^\omega$-systems in terms of \emph{phase polynomials} (or \emph{polynomial phase functions}). A structural description of this kind was concomitantly proved by Bergelson, Tao and Ziegler in \cite{BTZ} (see \cite[Theorem 1.19]{BTZ}). More precisely, they proved that the Host--Kra factor of order $k$ for ergodic $\mb{F}_p^\omega$-systems (for $p>k$) is a so-called \emph{Abramov system of order at most $k$}. Recall that an ergodic $\mb{F}_p^\omega$-system $\ns$ is said to be \emph{Abramov of order at most $k$} if the linear span of the phase polynomials of degree at most $k$ on $\ns$ is dense in $L^2(\ns)$ (see \cite{BTZ} or \cite[\S 5.2]{CGSS-p-hom}). In addition to this description in the high-characteristic case, in \cite{BTZ} it was also proved that in general the $k$-th Host--Kra factor of an ergodic $\mb{F}_p^\omega$-system is always an Abramov system of order $C$ for some constant $C$ depending only on $k$ (see \cite[Theorem 1.20]{BTZ}). This motivated the following conjecture (\cite[Remark 1.21]{BTZ}).

\begin{conjecture}\label{conj:BTZ}
For every ergodic $\mb{F}_p^\omega$-system and every $k\in \mb{N}$, the Host--Kra factor of order $k$ of this system is an Abramov system of order at most $k$.
\end{conjecture}
\noindent As mentioned above, the main result of \cite{BTZ}  established this conjecture for $k<p$. 

More recently, in \cite{CGSS-p-hom} a different approach using nilspace theory extended the confirmation of Conjecture \ref{conj:BTZ} to $k\leq p+1$. This involved \emph{nilspace systems} \cite[Definition 5.10]{CScouplings}. These systems form a class of measure-preserving (and also topological dynamical) systems which extends that of the classical nilsystems, and which is useful in the analysis of Host--Kra factors (see \cite[Theorem  5.11]{CScouplings} and \cite[Theorem 5.1]{CGSS}) as well as in topological dynamics \cite{GMV3}. The above-mentioned results from \cite{CGSS-p-hom} on Conjecture \ref{conj:BTZ} involved recasting the conjecture as a problem of describing certain nilspace systems based on a special class of compact nilspaces, called \emph{$p$-homogeneous nilspaces} (further discussed below).

Even more recently, by an interesting use of nilspace-theoretic cocycles, Jamneshan, Shalom and Tao produced a counterexample to Conjecture \ref{conj:BTZ} for $k=5$ and $p=2$ in \cite{JST2}. In their related paper \cite{JST1}, they posed the following question, which can be viewed as asking for the next best thing once Conjecture \ref{conj:BTZ} is known not to hold in general (see \cite[Question 1.6]{JST1}). 

\begin{question}\label{Q:JST}
Can every ergodic $\mb{F}_p^\omega$-system of order $k$ be extended \textup{(}within the category of $\mb{F}_p^\omega$-systems\textup{)} to an Abramov system of order at most $k$?    
\end{question}

\noindent The main result of this paper is the following affirmative answer to Question \ref{Q:JST}.

\begin{theorem}\label{thm:main-intro}
Let $(\ns,\mc{X},\lambda)$ be\footnote{Here, as usual, $\mc{X}$ denotes a $\sigma$-algebra on the set $\ns$, and $\lambda$ denotes a probability measure on $\mc{X}$.} an ergodic $\mb{F}_p^\omega$-system of order $k$. Then there exists a minimal distal topological dynamical $\mb{F}_p^\omega$-system $\nss$,  and an ergodic $\mb{F}_p^\omega$-invariant Borel probability measure $\mu$ on $\nss$, such that the $\mb{F}_p^\omega$-system $(\nss,\mc{Y},\mu)$ is Abramov of order at most $k$ and $(\ns,\mc{X},\lambda)$ is a factor of $(\nss,\mc{Y},\mu)$. 
\end{theorem}
\noindent As mentioned above, a central role in the nilspace approach to the analysis of $\mb{F}_p^\omega$-systems is played by $p$-homogeneous nilspaces \cite[Definition 1.2]{CGSS-p-hom}. In particular it was proved in \cite[Theorem 5.3]{CGSS-p-hom} that the Host--Kra factors for such systems are isomorphic to uniquely ergodic nilspace systems whose underlying nilspace is $p$-homogeneous. This is crucial for our proof of Theorem \ref{thm:main-intro}, which we now outline.

Using unique ergodicity, any given ergodic $\mb{F}_p^\omega$-system $\ns$ of order $k$ (i.e.\ isomorphic to its $k$-th Host--Kra factor) can be viewed as a system consisting of a single orbit closure $\overline{\orb(x_0)}$ in the space of nilspace morphisms $\hom(\mc{D}_1(\mb{F}_p^\omega),\ns)$ (this orbit is itself an $\mb{F}_p^\omega$-system under shifting of the morphism). Then, we use one of the main  results from \cite{CGSS-aff-ex} to obtain a nilspace fibration $\varphi$ from a ``universal" $k$-step $p$-homogeneous nilspace $\mc{H}_{p,k}$ onto $\ns$ (see \cite[Theorem 4.3]{CGSS-aff-ex}). We can then lift the orbit closure of $x_0$ through $\varphi$ to an orbit closure in $\hom(\mc{D}_1(\mb{F}_p^\omega),\mc{H}_{p,k})$. The fact that $\mc{H}_{p,k}$ is an \emph{abelian group-nilspace} (i.e.\ a nilspace based on a filtered abelian group) implies that the latter orbit is Abramov as a topological $\mb{F}_p^\omega$-system (hence also Abramov as a measure-preserving system). In fact, we show that this is an equivalent description of topological Abramov $\mb{F}_p^\omega$-system, namely, that every such system is the closure of the orbit of some polynomial in $\hom(\mc{D}_1(\mb{F}_p^\omega),\mc{H}_{p,k})$; see Theorem \ref{thm:AbramovChar} (the notion of topological Abramov systems is recalled in Section \ref{sec:homspace}).

To detail the above argument, we begin by gathering the required properties of systems of the form $\hom(\mc{D}_1(\mb{F}_p^\omega),\ns)$ in Section \ref{sec:homspace} below. We then focus on the case $\mc{H}_{p,k}$ in Section \ref{sec:CharSec}. The main result is then proved in Section \ref{sec:mainproof}.

\vspace*{0.5cm}

\noindent \textbf{Acknowledgements}. All authors used funding from project PID2020-113350GB-I00 from Spain’s MICINN.
The second-named author received funding from projects KPP 133921 and Momentum
(Lend\"ulet) 30003 of the Hungarian Government. The research was also supported partially
by the NKFIH ``\'Elvonal” KKP 133921 grant and partially by the Hungarian Ministry of
Innovation and Technology NRDI Office within the framework of the Artificial Intelligence
National Laboratory Program. We thank Asgar Jamneshan, Or Shalom and Terence Tao for useful feedback on this paper.

\section{\texorpdfstring{The space of nilspace morphisms $\hom(\mc{D}_1(\ab),\ns)$ as a subshift}{The space of nilspace morphisms hom(D1(Z),X) as a subshift}}\label{sec:homspace}
\begin{defn}
Let $\ab$ be a countable discrete abelian group and let $\ns$ be a compact $k$-step nilspace. We denote by $\hom(\mc{D}_1(\ab),\ns)$ the set of all nilspace morphisms $f:\mc{D}_1(\ab)\to \ns$.
\end{defn}
\noindent We equip the set $\hom(\mc{D}_1(\ab),\ns)$  straightaway, and for the rest of this paper, with a natural topology inherited from that of $\ns$, namely the subspace topology as a subset of $\ns^{\ab}$. Note that the topology on $\ns^{\ab}$ is compact Hausdorff, and since $\ab$ is countable, the topology is also second-countable (the topology on a compact nilspace is assumed to be second-countable by default \cite{Cand:Notes2}).

\begin{lemma}
The set $\hom(\mc{D}_1(\ab),\ns)$ is a closed subset of $\ns^{\ab}$, hence a compact space.
\end{lemma}

\begin{proof}
By definition of morphisms, we have $f\in \hom(\mc{D}_1(\ab),\ns)$ if and only if for every $\q\in \cu^n(\mc{D}_1(\ab))$ we have $f\co \q\in \cu^n(\ns)$. Thus we have
\[
\hom(\mc{D}_1(\ab),\ns) = \bigcap_{n=0}^\infty\; \bigcap_{\q\in \cu^n(\mc{D}_1(\ab))} \{f\in \ns^{\ab}: f\co \q \in \cu^n(\ns) \}.
\]
As the set $\cu^n(\ns)$ is closed, the set $\{f\in \ns^{\ab}: f\co \q \in \cu^n(\ns) \}$ is closed as well for any $\q$. Hence $\hom(\mc{D}_1(\ab),\ns)$ is closed.
\end{proof}

On this space $\hom(\mc{D}_1(\ab),\ns)$ we can define a natural action of $\ab$ by shifts, as follows.

\begin{defn}
 We define the following continuous action of $\ab$ on $\hom(\mc{D}_1(\ab),\ns)$:  for any $f\in \hom(\mc{D}_1(\ab),\ns)$ and $z\in \ab$, we let $(f,z)\mapsto f(\cdot+z)$, where $f(\cdot+z)$ denotes the map $x\mapsto f(x+z)$ in $\hom(\mc{D}_1(\ab),\ns)$.
\end{defn}
\noindent This construction has an important feature for what follows, namely, that if the $k$-step nilspace $\ns$ has a sufficiently simple structure (in particular, if it is an abelian group nilspace), then the dynamical system consisting of $\hom(\mc{D}_1(\ab),\ns)$ with the above $\ab$-action is a \emph{topological} Abramov system of order at most $k$. To detail this we first recall the latter notion in the definition below (from \cite[Definition 5.12]{CGSS-p-hom}). This uses the notation $C(\ns,\mb{C})$ for the algebra of continuous functions on a metric space $\ns$, and the notation $\Delta_g(f)$ for the multiplicative derivative of a function $f\in L^\infty(\ns)$ when $\ns$ is acted upon by a group $G$, defined by $\Delta_g f (x) := f (g\,x)\overline{f (x)}$ for any $g\in G, x\in \ns$ (in particular $\Delta_g f (x) = f (g+ x)\overline{f (x)}$ if $G$ is abelian).

\begin{defn}\label{def:top-abramov}
Let $Y$ be a compact metric space and let $\ab$ be a group acting on $Y$ by homeomorphisms. A continuous phase polynomial of degree at most $k$ is a  function $\phi\in C(Y,\mb{C})$ such that for every $y\in Y$ and $z_1,\ldots,z_{k+1}\in \ab$ we have $\Delta_{z_1}\cdots \Delta_{z_{k+1}} \phi(y) = 1$. We say that $(Y,\ab)$ is a \emph{topological Abramov system of order at most $k$} if the algebra of continuous phase polynomials of degree at most $k$ is dense in $C(Y,\mb{C})$.
\end{defn}

\begin{remark}
In particular, if $\mu$ is any Borel probability measure on $Y$ preserved by the action of $\ab$, then the above system $(Y,\ab)$ together with $\mu$ is an Abramov measure-preserving system of order at most $k$ (this is seen straighforwardly using the density of $C(Y,\mb{C})$ in $L^2(\nss)$).
\end{remark}

Let us now prove the feature announced above.

\begin{proposition}\label{prop:top-abr}
Let $G$ be a compact abelian group, let $G_\bullet$ be a filtration of degree $k$ on $G$, let $\ns$ be the associated group nilspace \textup{(}i.e.\ $\ns$ is $G$ equipped with the Host-Kra cubes associated with the filtration $G_\bullet$\textup{)}, and let $\ab$ be a countable discrete abelian group. Then the $\ab$-system  $\hom(\mc{D}_1(\ab),\ns)$ is a topological Abramov system of order at most $k$.
\end{proposition}

\begin{proof}
Abusing the notation, let us denote the nilspace $\ns$ simply by $G$. For any $\chi\in \widehat{G}$ and $z\in \ab$ let $h_{z,\chi}:\hom(\mc{D}_1(\ab),G) \to \mb{C}$ be the map $f\mapsto \chi(f(z))$. Letting $\nss:=\hom(\mc{D}_1(\ab),G)$, we claim that the algebra generated by these functions $h_{z,\chi}$ separates the points of $\nss$ and vanishes nowhere. Indeed, if $f,g$ are distinct points in $\nss$ then there exists $z\in \ab$ such that $f(z)\not=g(z)$, and then by standard Fourier analysis there exists $\chi\in \widehat{G}$ such that $\chi(f(z))\neq \chi(g(z))$ (the algebra clearly vanishes nowhere since any $h_{z,\chi}$ takes values in the unit circle). Our claim follows, whence the algebra generated by the functions $h_{z,\chi}$ is dense in $C(\nss,\mb{C})$, by the Stone-Weierstrass theorem.

Next we show that every map $h_{z,\chi}$ is a polynomial phase function of degree at most $k$. Thus, we want to prove that for any  $z_1,\ldots,z_{k+1}\in \ab$ and any $f\in \hom(\mc{D}_1(\ab),G)$ we have $\Delta_{z_{k+1}}\cdots\Delta_{z_1}h_{z,\chi}(f)=1$.  For any cube $\q\in \cu^{k+1}(\mc{D}_1(\ab))$ defined as $\q(v_1,\ldots,v_{k+1}):=z+v_1z_1+\cdots+v_{k+1}z_{k+1}$, we have $\Delta_{z_{k+1}}\cdots\Delta_{z_1}h_{z,\chi}(f) = \prod_{v\in \db{k+1}} \mc{C}^{|v|}\chi(f\co \q)(v)$, where $|v|=v_1+\cdots+v_{k+1}$ and $\mc{C}$ denotes the complex-conjugation operation (thus $\mc{C}^{|v|}$ is the complex conjugation if $|v|$ is odd and the identity if $|v|$ is even). Now, since  $f\co\q\in \cu^{k+1}(\ns)$,  we know that $\sum_{v\in \db{k+1}} (-1)^{|v|}(f\co \q)(v)=0$ by \cite[Proposition 2.2.28]{Cand:Notes1}. Composing with $\chi$ we deduce that $\Delta_{z_{k+1}}\cdots\Delta_{z_1}h_{z,\chi}(f)=1$ as required.
\end{proof}

\noindent Next we observe additional useful properties of the system $\hom(\mc{D}_1(\ab),\ns)$.

We begin with distality. Recall that a topological dynamical system $(Y,G)$ (where $Y$ is a compact metric space and $G$ is a countable discrete group acting on $Y$ by homeomorphisms) is said to be  \emph{distal} if no two distinct points $x \neq y$ in $Y$ are \emph{proximal} (where $x, y$ are said to be proximal if there is a sequence $(g_n)_{n\in \mb{N}}$ in $G$ such that $\lim_{n \to \infty} d(g_n\cdot x, g_n\cdot y) = 0$). Using the compactness of $Y$ to pass to convergent subsequences, it is readily seen that $(Y,G)$ is distal if and only if for every $x,y\in Y$ and every sequence $(g_n)_{n\in \mb{N}}$ in $G$, if   $g_n\cdot x$ and $g_n\cdot y$ both converge to the same limit as $n\to\infty$, then $x=y$.

\begin{lemma}\label{lem:hom-are-distal}
Let $\ab$ be a discrete countable abelian group, let $(G,G_\bullet)$ be a filtered compact abelian group of degree at most $k$, and let $\ns$ be the associated group nilspace. Then the $\ab$-system $\hom(\mc{D}_1(\ab),\ns)$ is distal.
\end{lemma}
\noindent We shall use the following notation: for abelian groups $G$, $\ab$ and any $f\in \hom(\mc{D}_1(\ab),G)$, for any $x,g\in \ab$ we let $\partial_g f(x):=f(x+g)-f(x)$.
\begin{proof}
Again we abuse notation, writing $G$ instead of $\ns$. It suffices to prove that if for some $f\in \hom(\mc{D}_1(\ab),G)$ and some sequence $(z_n)$ in $\ab$ we have $\lim_{n\to\infty}z_n(f)=\lim_{n\to\infty} f(\cdot+z_n)=0$ (here $0$ is the constant zero function in $\hom(\mc{D}_1(\ab),G)$) then $f=0$. 

First, note that for every $g_0,\ldots,g_{k},x\in \ab$ we have $\partial_{g_0}\cdots \partial_{g_k} f(x+z_n) = 0$ for every $n\in \mb{N}$. In particular this implies that $\partial_{g_1}\cdots \partial_{g_k} f(x+z_n) = c$ for some constant $c\in G$ independent of $z_n,x\in \ab$. On the other hand, by assumption $f(x+z_n)\to 0$ as $n\to\infty$, so we must have $c=0$. Hence for every $g_1,\ldots,g_k\in \ab$ we have $\partial_{g_1}\cdots \partial_{g_k} f(x+z_n) =0$. 

We can repeat the previous argument inductively to deduce that then $\partial_{g_2}\cdots \partial_{g_k} f(x+z_n) =0$, then $\partial_{g_3}\cdots \partial_{g_k} f(x+z_n) =0$, and so on, deducing finally that $f=0$.
\end{proof}

\noindent We will be particularly interested in single orbit closures inside $\hom(\mc{D}_1(\ab),\ns)$. 

\begin{corollary}
Let $\ab$ be a discrete countable abelian group, let $(G,G_\bullet)$ be a filtered compact abelian group of degree at most $k$, and let $\ns$ be the associated group nilspace. Let $f\in \hom(\mc{D}_1(\ab),\ns)$, and let  $\overline{\orb(f)}:=\overline{\{f(\cdot+g):g\in \ab\}}$ be the closure of the orbit of $f$. Then $(\overline{\orb(f)},\ab)$ is a minimal distal dynamical system. 
\end{corollary}
\begin{proof}
The point $f$ is proximal to a uniformly recurrent point \cite[Theorem 8.7]{Furstenberg}. Since the system is distal, the point $f$ is itself uniformly recurrent \cite[Corollary]{Furstenberg}. This then implies that the orbit closure of $f$ is minimal (see \cite[Proposition 2.4.5]{Hass&Katok}).  
\end{proof}
\noindent Using this, we can apply nilspace-theoretic results in topological dynamics from \cite{GGY} to prove that in fact $(\overline{\orb(f)},\ab)$ is a $k$-step nilspace system.

\begin{theorem}\label{thm:poly-gen-min-dis-ilspace-systems}
Let $\ab$ be a discrete countable abelian group, let $(G,G_\bullet)$ be a filtered compact abelian group of degree at most $k$, and let $\ns$ be the associated group nilspace. Then for every $f\in \hom(\mc{D}_1(\ab),\ns)$, the system $(\overline{\orb(f)},\ab)$ is a minimal distal nilspace system of order at most $k$, and an Abramov system of order at most $k$.
\end{theorem}

\begin{proof}
By \cite[Theorem 7.14]{GGY} it suffices to prove that $(\overline{\orb(f)},\ab)$ is a system of order $k$, i.e.\ that the regionally proximal relation of order $k$ (see \cite{Song-Ye}) is trivial. In our case this means proving the following: given any sequence $g^{(n)}:=(g_1^{(n)},\ldots,g_{k+1}^{(n)}) \in \ab^{k+1}$, if for every $v\in \db{k+1}\setminus\{1^{k+1}\}$ we have $f(x+v\cdot g^{(n)})\to f(x)$ as $n\to \infty$ (where $v\cdot g^{(n)}=v_1g^{(n)}_1+\cdots+v_{k+1}g^{(n)}_{k+1}$), then $f(x+1^{k+1}\cdot g^{(n)})\to f(x)$ as $n\to\infty$ (where $1^{k+1}\in \db{k+1}$ is the element with all coordinates equal to 1). This follows from the fact that  $\partial_{g_1^{(n)}}\cdots \partial_{g_{k+1}^{(n)}} f(x) = 0$ for every $x\in \ab$ and $n\in \mb{N}$. Indeed, this fact implies that for every $n$ we have $f(x+1^{k+1}\cdot g^{(n)}) = (-1)^k\sum_{v\in \db{k+1}\setminus\{1^{k+1}\}}(-1)^{|v|} f(x+v\cdot g^{(n)})$ (where $|v|:=\sum_{i\in [k+1]}v_i$). Letting $n\to\infty$, all terms in the right side in the previous equation converge to $f(x)$, whence $f(x+1^{k+1}\cdot g^{(n)}) \to   (-1)^k \sum_{v\in \db{k+1}\setminus\{1^{k+1}\}} (-1)^{|v|}f(x) = (-1)^{2k}f(x)=f(x)$, as required. 

Finally, the claim that $(\overline{\orb(f)},\ab)$ is topological Abramov of order at most $k$ follows from (the proof of) Proposition \ref{prop:top-abr}. In fact, that proposition shows that $\hom(\mc{D}_1(\ab),\ns)$ is topological Abramov of order $k$ by constructing  functions $h_{z,\chi}$ that separate the points and vanish nowhere, but then the restrictions $h_{z,\chi}|_{\overline{\orb(f)}}$ also separate the points and vanish nowhere on $\overline{\orb(f)}$. The claim then follows by the Stone-Weierstrass theorem.
\end{proof}

\section{\texorpdfstring{Transitive topological Abramov $\mb{F}_p^\omega$-systems as  single orbit closures in $\hom(\mc{D}_1(\mb{F}_p^\omega),\mc{H}_{p,k})$}{Transitive topological Abramov Fp⍵-systems as single orbit closures in hom(D1(Fp⍵),$H_{p,k})$}
}\label{sec:CharSec}

\noindent In the previous section we identified some useful properties of dynamical systems consisting of single orbit closures of morphisms into an abelian nilspace $\ns$. In this section we focus on the case where $\ns$ is the $p$-homogeneous nilspace $\mc{H}_{p,k}$, which already played an important role in previous work (e.g.\ in \cite{CGSS-aff-ex}), and whose definition is recalled below. This special case is of interest in that, as we shall prove in this section, it turns out that every topological Abramov $\mb{F}_p^\omega$-system is the closure of a single orbit in $\hom(\mc{D}_1(\mb{F}_p^\omega),\mc{H}_{p,k})$ (see Theorem \ref{thm:AbramovChar} below), which yields a pleasant description of these Abramov systems. To explain this we begin by recalling the definition of these $p$-homogeneous nilspaces $\mc{H}_{p,k}$ from \cite[Definition 4.2]{CGSS-aff-ex}. 

\begin{defn}
For any prime $p\in \mb{N}$, let $\mf{Z}_p$ denote the group of $p$-adic integers. For $\ell\in\mb{N}$ we denote by $\abph^{(p)}_{\infty,\ell}$ the compact group-nilspace associated with the filtered group  $\big(G, G_\bullet^{(\ell)}\big)$ where $G= \mf{Z}_p$ and the filtration $G_\bullet^{(\ell)}=(G_{(i)}^{(\ell)})_{i=0}^{\infty}$ consists of $G_{(i)}^{(\ell)}= \mf{Z}_p$ for $i\in[0,\ell]$ and $G_{(i)}^{(\ell)}=p^{\lfloor\frac{i-{\ell}-1}{p-1}\rfloor+1}\mf{Z}_p$ for $i>\ell$. We define $\mc{H}_p$ as the product nilspace $\prod_{\ell=1}^{\infty}(\abph^{(p)}_{\infty,\ell})^{\mb{N}}$. We denote its $k$-th order factor by $\mc{H}_{p,k}$.
\end{defn}

\noindent The nilspace $\mc{H}_{p,k}$ can also be described in terms of products of finite nilspaces. Indeed, let $\ell$ be a positive integer at most $k$, let $G$ be the cyclic group of order $p^{\lfloor \frac{k-{\ell}}{p-1}\rfloor+1}$, and let us equip this with  the filtration $G_\bullet=(G_{(i)})_{i=0}^{\infty}$ with $i$-th term $G_{(i)}=G$ for $i\in [0,\ell]$, and $G_{(i)}:=p^{\lfloor\frac{i-{\ell}-1}{p-1}\rfloor+1}G=\{p^{\lfloor\frac{i-{\ell}-1}{p-1}\rfloor+1}x:x\in G\}$ for $i>\ell$. The nilspace $\abph_{k,{\ell}}^{(p)}$ (see \cite[Definition 1.6]{CGSS-p-hom}) is then the group nilspace associated with the filtered group $(G,G_\bullet)$. In particular, $\mc{H}_{p,k}$ can be shown to be equal to\footnote{ Note that each term $\abph^{(p)}_{\infty,\ell}$ is a group nilspace $(G,G_\bullet^{(\ell)})$ and thus, its $k$-th order factor equals $G/G_{(k)}^{(\ell)}$. }  $\prod_{\ell=1}^{k}(\abph^{(p)}_{k,\ell})^{\mb{N}}$.

The main result of this section is then the following.

\begin{theorem}\label{thm:AbramovChar} 
Let $\ns$ be a compact second-countable topological space acted upon continuously by $\mb{F}_p^\omega$. The following statements are equivalent:
\setlength{\leftmargini}{0.8cm}
\begin{enumerate}
    \item The $\mb{F}_p^\omega$-system $\ns$ is a transitive topological Abramov system of order at most $k$.
    \item There exists a morphism \textup{(}i.e.\ a polynomial\textup{)} $f\in \hom(\mc{D}_1(\mb{F}_p^\omega),\mc{H}_{p,k})$ such that $\ns$ is conjugate \textup{(}i.e.\ isomorphic as a topological dynamical system\textup{)} to $(\overline{\orb(f)},\mb{F}_p^\omega)$.
\end{enumerate}
\end{theorem}
\begin{proof}
$(ii)\Rightarrow (i)$: this follows from Theorem \ref{thm:poly-gen-min-dis-ilspace-systems} applied with $\ab=\mb{F}_p^\omega$ (in fact by previous results we know that here we obtain more than $(i)$, namely we obtain that $\ns$ is then minimal and distal.)

$(i)\Rightarrow (ii)$: by the assumption and the separability of $C(\ns,\mb{C})$ (see \cite[Section 12.3 page 251]{Royden-Fitzpatrick}), there is a \emph{sequence} $(g_i)_{i\in \mb{N}}$ of (continuous) polynomial phase functions whose finite linear combinations form a dense set in $C(\ns,\mb{C})$. For each $i$ let $\phi_i:\ns\to \mb{T}$ be a continuous $\mb{F}_p^\omega$-polynomial map such that $g_i(x)=e(\phi_i(x))$. This implies that there is a continuous function $F:\mb{T}^\mb{N}\to \mb{C}^\mb{N}$ (whose components consist in those linear combinations) such that for every $h\in C(\ns,\mb{C})$ and $\epsilon>0$ there is $N\in\mb{N}$ such that $|h(x)-F((\phi_i(x))_{i\in\mb{N}})_N|<\epsilon$ for all $x\in\ns$.

By transitivity, there exists $x_0\in \ns$ such that $\ns$ is isomorphic (as a dynamical system) to the orbit closure $\overline{\orb(x_0)}$.  Consider $(\phi_i(x_0))_{i\in \mb{N}}\in \mb{T}^\mb{N}$, and note that the map $f:\mb{F}_p^\omega\to \mb{T}^\mb{N}$ defined by $z\mapsto (\phi_i(z\cdot x_0))_{i\in \mb{N}}$ is a morphism in $\hom(\mc{D}_1(\mb{F}_p^\omega),\mc{D}_k(\mb{T}^\mb{N}))$. We let $\mb{F}_p^\omega$ act on $f$ as a single point in the subshift $\hom(\mc{D}_1(\mb{F}_p^\omega),\mc{D}_k(\mb{T}^\mb{N}))$, and consider the corresponding orbit $\orb(f)$. Let $\theta:\orb(x_0)\to \orb(f)$ be the map $x\mapsto (z\mapsto (\phi_i(z\cdot x))_{i\in\mb{N}}$. By definition of $f$, the map $\theta$ is $\mb{F}_p^\omega$-equivariant. This map is also injective, for if $x\neq x'$ in $\orb(x_0)$, then there is $h\in C(\ns,\mb{C})$ with $h(x)\neq h(x')$ (by Urysohn's lemma \cite[Theorem 33.1]{Munkres}) and therefore there is $N$ such that $F((\phi_i(x))_{i\in\mb{N}})_N\neq F((\phi_i(x'))_{i\in\mb{N}})_N$ by the above construction of $F$, so $(\phi_i(x))_{i\in\mb{N}}\neq (\phi_i(x'))_{i\in\mb{N}}$ whence $\theta(x)\neq \theta(x')$. The map $\theta$ is also continuous, and therefore extends to an injective continuous map $\overline{\orb(x_0)}\to\overline{\orb(f)}$, which we denote also by $\theta$. By compactness, it follows that $\theta$ is a homeomorphism $\overline{\orb(x_0)}\to\overline{\orb(f)}$, so $\overline{\orb(x_0)}$ and  $\overline{\orb(f)}$ are indeed conjugate $\mb{F}_p^\omega$-systems. 

It remains only to check that $f$ can in fact be assumed to take values in $\mc{H}_{p,k}$ (viewed as a subset of $\mb{T}^\mb{N}$).
This follows from results in \cite{CGSS-aff-ex}. To detail this, let us define for every $n\in \mb{N}$ the map $\gamma_n\in \hom(\mc{D}_1(\mb{F}_p^n),\mc{D}_1(\mb{F}_p^\omega))$ by $\gamma_n(v)=(v,0^{\omega\setminus[n]})$. For any $j\in \mb{N}$ let $\pi_j:\mb{T}^\mb{N}\to \mb{T}$ denote the projection to the $j$-th coordinate. Note that $\pi_j\co f\co \gamma_n\in \hom(\mc{D}_1(\mb{F}_p^n),\mc{D}_k(\mb{T}))$. By \cite[Proposition E.3]{CGSS-aff-ex}, we have that $\pi_j\co f\co \gamma_n-\pi_j\co f\co \gamma_n(0^{\mb{N}})$ has its image in $\abph_{k,\ell}$ for some $\ell\in[p-1]$. Letting $n\to \infty$ it is readily seen that $\pi_j\co f-\pi_j\co f(0^\mb{N})$ has its image inluded in $\abph_{k,\ell}$. The topological dynamical system $\overline{\orb(f)}$ is conjugate of $\overline{\orb(f-(\pi_j(f(0^\mb{N})))_{j\in \mb{N}})}$ and the latter is clearly isomorphic to the closure of an element in $\hom(\mc{D}_1(\mb{F}_p^\omega),\mc{H}_{p,k})$.
\end{proof}

\section{Proving the main result}\label{sec:mainproof}
\noindent In this section we prove Theorem \ref{thm:main-intro}. To this end we first gather the main ingredients, the first of which is \cite[Theorem 1.9]{CGSS-p-hom}, which we recall here.

\begin{theorem}\label{thm:erg-factors-as-nilspace-sys}
For every $k\in  \mb{N}$, the $k$-th Host-Kra factor of an ergodic $\mb{F}_p^\omega$-system is isomorphic to a $p$-homogeneous $k$-step nilspace system $(\ns,\mb{F}_p^\omega)$.
\end{theorem}
\noindent In particular, any ergodic $\mb{F}_p^\omega$-system of order $k$ is isomorphic to a nilspace system $(\ns,\mb{F}_p^\omega)$. This result allow us to reduce the problem to the study of nilspace systems. 

The following is a quick  consequence of \cite[Theorem 4.3]{CGSS-aff-ex}.

\begin{theorem}\label{thm:covering}
For any $p$ prime and any $k\in \mb{N}$ the following holds. Let $\ns$ be any compact $k$-step $p$-homogeneous nilspace. Then there exists a fibration $\varphi:\mc{H}_{p,k}\to \ns$.
\end{theorem}

\begin{proof}
The result \cite[Theorem 4.3]{CGSS-aff-ex} is stronger, indeed it shows that for every profinite-step $p$-homogeneous nilspace $\ns$  there is a fibration $\mc{H}_p\to\ns$. If $\ns$ is $k$-step then the fibration factors through the $k$-step canonical nilspace factor of $\mc{H}_p$, which is precisely $\mc{H}_{p,k}$. The result follows.
\end{proof}

The final ingredient that we need is the following.

\begin{proposition}\label{prop:lifting}
Let $\ns,\nss$ be nilspaces and let $\varphi:\nss\to \ns$ be a fibration. Then, for every $f\in \hom(\mc{D}_1(\mb{Z}^\omega),\ns)$ there exists $g\in \hom(\mc{D}_1(\mb{Z}^\omega),\nss)$ such that $\varphi \co g = f$.
\end{proposition}

\begin{proof}
The proof of this result is essentially the same as the proof of \cite[Corollary A.6]{CGSS-p-hom}. In fact, the latter result establishes the similar claim but replacing $\mb{Z}^\omega$ with $\mb{Z}^n$ for some fixed $n$. The proof consists in proving that  we can define the values of $g$ inductively on sets of the form $[-L,L]^n\subset \mb{Z}^n$ for $L$ increasing (eventually defining a map on the whole of $\mb{Z}^n$). In $\mb{Z}^\omega$ we can carry out the same inductive argument but defining the values of $g$ in sets of the form $[-L,L]^L\times \{0\}^{\mb{N}\setminus[L]}\subset \mb{Z}^\omega$ for increasing values of $L$. Since $\bigcup_{L=1}^\infty [-L,L]^L\times \{0\}^{\mb{N}\setminus[L]} = \mb{Z}^\omega$, the result follows.
\end{proof}

As a consequence of these results we have the following lemma.

\begin{lemma}\label{lem:lift}
Let $\ns$ be any compact $k$-step $p$-homogeneous nilspace, and let $\varphi:\mc{H}_{p,k}\to \ns$ be the fibration from Theorem \ref{thm:covering}. Then for any $f\in \hom(\mc{D}_1(\mb{F}_p^\omega),\ns)$ there exists $g\in \hom(\mc{D}_1(\mb{F}_p^\omega),\mc{H}_{p,k})$ such that $\varphi \co g = f$.
\end{lemma}

\begin{proof}
We can consider any such $f\in \hom(\mc{D}_1(\mb{F}_p^\omega),\ns)$ as an element $\tilde{f}\in  \hom(\mc{D}_1(\mb{Z}^\omega),\ns)$ simply by composing with the natural quotient map $\mb{Z}^\omega\to \mb{F}_p^\omega$. By Proposition \ref{prop:lifting}, there exists $\tilde{g}\in \hom(\mc{D}_1(\mb{Z}^\omega),\mc{H}_{p,k})$ such that $\varphi\co \tilde{g} = \tilde{f}$. Since $\mc{H}_{p,k}$ is $p$-homogeneous, if we restrict $\tilde  g$ to the usual fundamental domain of the homomorphism $\mb{Z}^\omega\to \mb{F}_p^\omega$ and we identify that fundamental domain $F$ with $\mb{F}_p^\omega$, we claim that this restriction equals some $g\in \hom(\mc{D}_1(\mb{F}_p^\omega),\mc{H}_{p,k})$. To be more precise, let $F\subset \mb{Z}^\omega$ be $\{(n_1,n_2,\ldots)\in \mb{Z}^\omega: \forall i\in \mb{N}, \; n_i\in[0,p-1]\}$ and consider the bijective map $\iota:F\to \mb{F}_p^\omega$ that sends $(n_1,n_2,\ldots)\mapsto (n_1\mod p,n_2 \mod p,\ldots)$. We claim that $g:=\tilde{g}\co \iota^{-1}$ is in $\hom(\mc{D}_1(\mb{F}_p^\omega),\mc{H}_{p,k})$.

To prove this note that it would suffice to prove that for any $\q\in \cu^n(\mc{D}_1(\mb{F}_p^\omega))$ we need to prove that $g\co \q\in \cu^n(\mc{H}_{p,k})$. For any such fixed $\q$, there exists an integer $N=N(\q)\ge 0$ such that $\q(\db{n})\subset \mb{F}_p^N$. If we denote by $\iota_N:[0,p-1]^N\times\{0\}^{\mb{N}\setminus[N]}\to \mb{F}_p^N$ note that $g\co \q = \tilde{g}\co \iota^{-1}\co \q = \tilde{g} \co \iota_N^{-1}\co \q = (\tilde{g}|_{\mb{Z}^N\times \{0\}^{\mb{N}\setminus\{0\}^{[N]}}}) \co \iota_N^{-1}\co \q$. But now $(\tilde{g}|_{\mb{Z}^N\times \{0\}^{\mb{N}\setminus\{0\}^{[N]}}}) $ is an element of $\hom(\mc{D}_1(\mb{Z}^N),\mc{H}_{p,k})$ and $\mc{H}_{p,k}$ is $p$-homogeneous. Thus by definition of $p$-homogeneous nilspaces,\footnote{See \cite[Definition 1.2]{CGSS-p-hom}.} $(\tilde{g}|_{\mb{Z}^N\times \{0\}^{\mb{N}\setminus\{0\}^{[N]}}}) \co \iota_N^{-1}$ is an element of $ \hom(\mc{D}_1(\mb{F}_p^N),\mc{H}_{p,k})$. Hence $g\co \q\in \cu^n(\mc{H}_{p,k})$ and as this was for an arbitrary $\q$ we have that $g\in \hom(\mc{D}_1(\mb{F}_p^\omega),\mc{H}_{p,k})$. By definition we have that $\varphi\co g = f$ and the result follows.
\end{proof}

\noindent We are ready to prove the main result of this paper, Theorem \ref{thm:main-intro}. Let us restate this theorem here, in a slightly more refined form.

\begin{theorem}\label{thm:main}
Let $(\ns,\mc{X},\lambda)$ be an ergodic $\mb{F}_p^\omega$-system of order $k$. Then there exists a minimal distal $k$-step $\mb{F}_p^\omega$-nilspace system $\nss$, which is topological Abramov of order at most $k$, such that for any $\mb{F}_p^\omega$-invariant measure $\mu$ on the Borel $\sigma$-algebra $\mc{Y}$ on $\nss$, the measure-preserving $\mb{F}_p^\omega$-system $(\ns,\mc{X},\lambda)$ is a factor of the $\mb{F}_p^\omega$-system $(\nss,\mc{Y},\mu)$.
\end{theorem}
\noindent This theorem could be made more concise if orbit closures in $\hom(\mc{D}_1(\mb{F}_p^\omega),\mc{H}_{p,k})$ were known to be always \emph{uniquely ergodic} $\mb{F}_p^\omega$-systems; see the discussion in the next section around Question \ref{Q:UniqueErgo}.
\begin{proof}
By \cite[Theorem 5.3]{CGSS-p-hom}, the $k$-th Host-Kra factor of an ergodic $\mb{F}_p^\omega$-system is isomorphic to a nilspace system $(\ns,\mb{F}_p^\omega)$ that is uniquely ergodic. This implies in particular that for any point $x_0\in \ns$, this point is transitive and thus $\overline{\orb(x_0)}=\overline{\{z(x_0)\in\ns: z\in \mb{F}_p^\omega\}}=\ns$.

Next, consider the map $\phi^*:
\ns \to  \hom(\mc{D}_1(\mb{F}_p^\omega),\ns)$, 
defined as $x \mapsto \phi^*(x)$ where $\phi^*(x)(z):=z(x)$ for any $z\in \mb{F}_p^\omega$. This map is continuous. Indeed, if $x_n\to x$, by the continuity of the action of $\mb{F}_p^\omega$, for any fixed $z\in \mb{F}_p^\omega$ we have $z(x_n)\to z(x)$. Moreover $\phi^*$ is equivariant, i.e.\ for any $z'\in \mb{F}_p^\omega$ we have $\phi^*(z'(x))(z)=z(z'(x)) = z'(z(x)) = z'(\phi^*(x)(z))$. As $\ns$ and $\hom(\mc{D}_1(\mb{F}_p^\omega),\ns)$ are compact metric spaces, we have that $\phi^*$ is a homeomorphism $\ns\to \phi^*(\ns)$ \cite[Theorem 26.6]{Munkres}. This together with the equivariance of $\phi^*$ implies that  $\ns$ and $\phi^*(\ns)$ are isomorphic as topological dynamical systems. We shall now prove that, as a \emph{topological} dynamical system, the system $\phi^*(\ns)$ is a factor of a minimal distal $\mb{F}_p^\omega$-nilspace system that is Abramov as claimed, and we shall thereafter obtain the measure-theoretic part of the conclusion of the theorem.

By the first paragraph of this proof, fixing any  $x_0\in \ns$, we have that $\phi^*(\ns)$ is the closure of the orbit of $\phi^*(x_0)$. By Theorem \ref{thm:covering} there exists a fibration $\varphi:\mc{H}_{p,k}\to \ns$. This induces the map $\varphi^*:\hom(\mc{D}_1(\mb{F}_p^\omega),\mc{H}_{p,k}) \to \hom(\mc{D}_1(\mb{F}_p^\omega),\ns)$ defined as $\varphi^*(f) = \varphi\co f$. Thus, we have the following diagram:
\begin{center}
\begin{tikzpicture}
  \matrix (m) [matrix of math nodes,row sep=2em,column sep=4em,minimum width=2em]
  {
      & \hom(\mc{D}_1(\mb{F}_p^\omega),\mc{H}_{p,k})  \\
     \ns & \hom(\mc{D}_1(\mb{F}_p^\omega),\ns). \\};
  \path[-stealth]
    (m-1-2) edge node [right] {$\varphi^*$} (m-2-2)
    (m-2-1) edge node [above] {$\phi^*$} (m-2-2);
\end{tikzpicture}
\end{center}
Note that $\varphi^*$ is continuous and equivariant. By Lemma \ref{lem:lift}, given the element $\phi^*(x_0)\in \hom(\mc{D}_1(\mb{F}_p^\omega),\ns)$ there exists $g\in \hom(\mc{D}_1(\mb{F}_p^\omega),\mc{H}_{p,k})$ such that $\varphi^*(g)=\phi^*(x_0)$.

Now let $\nss$ be the $\mb{F}_p^\omega$-system $\overline{\orb(g)}=\overline{\{g(\cdot+z):z\in \mb{F}_p^\omega\}}$. We have $\varphi^*(\nss)=\phi^*(\ns)$. To see this, note that $\orb(g)\subset \overline{\orb(g)}$, apply $\varphi^*$ and note that, since $\overline{\orb(g)}$ is compact, $\varphi^*(\overline{\orb(g)})$ is closed. In particular $\overline{\varphi^*(\orb(g))}\subset \varphi^*(\overline{\orb(g)})$. But $\overline{\varphi^*(\orb(g))}\supset \overline{\orb(\phi^*(x_0))}=\phi^*(\ns)$. To prove the opposite inclusion, fix any $g'\in \nss$ and let $(z_n)_{n\in \mb{N}}$ be a sequence in $\mb{F}_p^\omega$ such that $g(\cdot+z_n)\to g'$ in $\hom(\mc{D}_1(\mb{F}_p^\omega),\mc{H}_{p,k})$. Applying $\varphi^*$ we find that $\phi^*(x_0)(\cdot+z_n) = \varphi^*(g)(\cdot+z_n)$, which  equals $\varphi^*(g(\cdot+z_n))$ by the equivariance of $\varphi^*$, and this converges to $\varphi^*(g')$ by continuity of $\varphi^*$. Hence $\varphi^*(g')\in \overline{\orb(\phi^*(x_0))}$.

By Theorem \ref{thm:AbramovChar} the system  $\nss$ is a topological Abramov system of order at most $k$, and it is a minimal distal nilspace system by Theorem \ref{thm:poly-gen-min-dis-ilspace-systems}. 

Finally, note that if $\mu$ is any $\mb{F}_p^\omega$-invariant Borel probability measure on $\nss$, then since the image measure $\mu\co(\varphi^*)^{-1}$ is an $\mb{F}_p^\omega$-invariant Borel probability measure on $\ns$, by the unique ergodicity of $\ns$ we must have $\mu\co(\varphi^*)^{-1}=\lambda$.
\end{proof}

\section{Final remarks and open questions}

\noindent The subshift $\hom(\mc{D}_1(\mb{F}_p^\omega),\mc{H}_{p,k})$ has played a central role in this paper. In fact, what we have proved shows that this subshift plays a ``universal" role in the following sense.
\begin{corollary}
Every ergodic $\mb{F}_p^\omega$-system of order $k$ is a factor of an orbit closure \textup{(}equipped with an invariant probability measure\textup{)} in the subshift $\hom(\mc{D}_1(\mb{F}_p^\omega),\mc{H}_{p,k})$. 
\end{corollary}
\noindent It thus seems of interest to study further this dynamical system $\hom(\mc{D}_1(\mb{F}_p^\omega),\mc{H}_{p,k})$ and its minimal subsystems. In particular, as mentioned after Theorem \ref{thm:main}, it would be interesting to know if such a minimal subsystem is always uniquely ergodic. 
\begin{question}\label{Q:UniqueErgo}
Is every orbit closure in the subshift $\hom(\mc{D}_1(\mb{F}_p^\omega),\mc{H}_{p,k})$ uniquely ergodic?
\end{question}
\noindent There are positive results in similar directions, for example it is known that an orbit closure in $\hom(\mc{D}_1(\mb{Z}),\mc{D}_k(\mb{T}))$ is always uniquely ergodic; see \cite{Glasner,Salehi}.

This question on unique ergodicity can be studied also in more general settings, for instance, for subshifts of the form $\hom(\ab,\ns)$ where $\ab$ is a countable discrete abelian group and $\ns$ is a compact abelian group-nilspace (as in Section \ref{sec:homspace}). 

Widening the scope further, it would be interesting to shed light on the possible unique ergodicity of certain more general classes of nilspace systems. In relation to this, it would help to clarify further the relations between the measure-preserving and the purely topological viewpoints when working with such systems. Indeed, note that in order to prove Theorem \ref{thm:main}, we first started with a measure-preserving system, then using nilspace theory we  viewed it as a topological dynamical system (Theorem \ref{thm:erg-factors-as-nilspace-sys}), then we solved in that setting the problem of finding an extension which is \emph{topologically Abramov} (see Definition \ref{def:top-abramov}) and finally we went back to the measure-preserving setting (Theorem \ref{thm:main}). In the more classical and better-known case of nilsystems, a deeper connection between the topological and measure-preserving viewpoints is known. For example, we have the following result from \cite[Ch.\ 11, Sec.\ 2, Theorem 11]{HKbook}, which captures several previous results and summarizes this connection.
\begin{theorem}\label{thm:nilsmeastop}
Let $(X,\mu,T)$ be a nilsystem where $\mu$ is the Haar measure on $\ns$. Then the following properties are equivalent:
\begin{enumerate}
    \item The nilsystem $(X,\mu,T)$ is ergodic.
    \item The topological nilsystem $(X,T)$ is uniquely ergodic.
    \item The topological nilsystem $(X,T)$ is minimal.
    \item The topological nilsystem $(X,T)$ is topologically transitive.
\end{enumerate}
\end{theorem}
\noindent This raises the question of whether an analogous result holds for nilspace systems. To discuss this in more detail we first present the following proposition, which goes some way toward obtaining such an analogous result.

\begin{proposition}\label{prop:topmeasconn}
Let $\ns$ be a $k$-step compact nilspace, let $G$ be a countable discrete nilpotent group, let $\phi:G\to \tran(\ns)$ be a homomorphism, and let $\mu_{\ns}$ be the Haar measure on $\ns$ \textup{(}see \cite[\S 2.2.2]{Cand:Notes2}\textup{)}. Each of the following properties implies the next one:
\begin{enumerate}
    \item The measure-preserving nilspace system $(\ns,\mu_{\ns},G,\phi)$ is ergodic.
    \item The topological nilspace system $(\ns,G,\phi)$ is uniquely ergodic,
    \item The topological nilspace system $(\ns,G,\phi)$ is minimal.
    \item The topological nilspace system $(\ns,G,\phi)$ is topologically transitive.
\end{enumerate}
\end{proposition}
\begin{proof}
The implication $(i)\implies (ii)$ follows from \cite[Lemma 5.6]{CGSS-p-hom} (where it is actually proved that for every integer $n\geq 0$, if $(\cu^n(\ns),\cu^n(G),\mu_{\cu^n(\ns)})$ is ergodic, then it is uniquely ergodic). Let us give  a short alternative proof of this implication here.

Suppose that $(i)$ holds, thus $\phi(G)\subset \tran(\ns)$ acts ergodically on $\ns$, relative to the Haar measure $\mu_{\ns}$ on $\ns$. Let $\nu$ be any $\phi(G)$-invariant Borel probability measure on $\ns$. We will prove that $(i)$ implies $(ii)$ by induction on $k$. The case $k=1$ follows from classical results, see e.g. \cite[Proposition 4.3.3]{Hass&Katok}. For $k>1$, we first observe that by induction on $k$ we can assume that $\nu\co \pi_{k-1}^{-1}=\mu_{\ns_{k-1}}$. Indeed, if $\mu_{\ns}$ is ergodic then $\mu_{\ns_{k-1}}=\mu_{\ns}\co \pi_{k-1}^{-1}$ is also ergodic relative to the action of $\widehat{\pi_{k-1}}(\phi(G))$.\footnote{See \cite[Lemma 1.5]{CGSS}.} In particular, as $\nu\co \pi_{k-1}^{-1}$ is preserved under the action of $\widehat{\pi_{k-1}}(\phi(G))$ and $(\ns_{k-1},\widehat{\pi_{k-1}}(\phi(G)))$ is uniquely ergodic by induction on $k$, we conclude that $\nu\co\pi_{k-1}^{-1} = \mu_{\ns_{k-1}}$.

Now for any $h\in\ab_k(\ns)$ let $V_h:\ns\to \ns$ be the map $V_h(x):=x+h$, (where addition here denotes the action of the structure group $\ab_k(\ns)$ on $\ns$). Letting $\mu_{\ab_k(\ns)}$ be the Haar probability measure on the compact abelian group $\ab_k(\ns)$, consider the Borel measure $\xi:=\int_{\ab_k(\ns)}\nu\co V_h\;d\mu_{\ab_k(\ns)}(h)$ on $\ns$ (i.e.\ defined by $\xi(A):=\int_{\ab_k(\ns)}\nu(A+h)\;d\mu_{\ab_k(\ns)}(h)$). Since  translations in nilspaces commute with the action of $ \ab_k(\ns)$ (by \cite[Lemma 3.2.37]{Cand:Notes1}) we conclude that for any $h\in \ab_k(\ns)$, the measure $\nu\co V_h$ is preserved by the action of $\phi(G)$, and therefore so is $\xi$. A simple calculation also shows that $\xi\co\pi_{k-1}^{-1}=\mu_{\ns_{k-1}}$. Furthermore, the measure $\xi$ is invariant under the action of $V_h$ for any $h\in \ab_k(\ns)$. Thus, by \cite[Proposition 2.2.5]{Cand:Notes2} we conclude that in fact $\xi = \mu_{\ns}$. But we have assumed that $\mu_{\ns}$ is ergodic, and we have written it as a convex combination $\mu_{\ns}=\int_{\ab_k(\ns)} \nu \co V_h\;d\mu_{\ab_k(\ns)}(h)$. As ergodic measures are extreme points in the convex set of invariant measures, we conclude that for almost every $h\in \ab_k(\ns)$ we have $\nu\co V_h=\mu_{\ns}$. Thus $\nu=\mu_{\ns}$ and $(ii)$ follows.\footnote{A similar argument shows that for any $n\in \mb{N}$, if $(\cu^n(\ns),\mu_{\cu^n(\ns)},\cu^n(G))$ is ergodic then $(\cu^n(\ns),\cu^n(G))$ is uniquely ergodic.}

To see the implication $(ii)\implies (iii)$, note that if $(\ns,G,\phi)$ is uniquely ergodic, then by standard arguments the support of the unique $G$-invariant measure is a minimal set (see e.g.\ \cite[Proposition 6.2.1]{O&V}). The unique invariant measure must be the Haar measure $\mu_{\ns}$, since this is $G$-invariant by \cite[Lemma 2.2.6]{Cand:Notes2}. Since $\mu_{\ns}$ is also a strictly positive measure \cite[Proposition 2.2.11]{Cand:Notes2}, the support of $\mu_{\ns}$ is $\ns$, so $\ns$ is minimal.

The implication $(iii)\implies (iv)$ is trivial. 
\end{proof}
\noindent Given Proposition \ref{prop:topmeasconn}, to obtain a full generalization of Theorem \ref{thm:nilsmeastop} for these nilspace systems, it would suffice to prove that property $(iv)$ above always implies property $(i)$.

\begin{question}
Let $(\ns,G,\phi)$ be a compact nilspace system as above. If $(\ns,G,\phi)$ is topologically transitive, must $(\ns,\mu_{\ns},G,\phi)$ be ergodic?
\end{question}


\begin{thebibliography}{1}

\bibitem{ABB} E. Ackelsberg, V. Bergelson, A. Best, \emph{Multiple recurrence and large intersections for abelian group actions}, 
Discrete Anal. (2021), Paper No. 18, 91 pp.

\bibitem{ABS} E. Ackelsberg, V. Bergelson, O. Shalom, \emph{Khintchine-type recurrence for 3-point configurations}, Forum Math. Sigma \textbf{10} (2022), Paper No. E107, 57 pp.

\bibitem{BTZ} V. Bergelson, T. Tao, T. Ziegler, \emph{An inverse theorem for the uniformity seminorms associated with the action of $\mb{F}_p^{\infty}$}, Geom. Funct. Anal. \textbf{19} (2010), no. 6, 1539--1596.

\bibitem{BTZ2} V. Bergelson, T. Tao, T. Ziegler, \emph{Multiple recurrence and convergence results associated to $\mb{F}_p^{\omega}$-actions}, J. Anal. Math. \textbf{127} (2015), 329--378.

\bibitem{Cand:Notes1} P. Candela, \emph{Notes on nilspaces: algebraic aspects}, Discrete Anal., 2017, Paper No. 15, 59 pp.

\bibitem{Cand:Notes2} P. Candela, \emph{Notes on compact nilspaces}, Discrete Anal., 2017, Paper No. 16, 57pp.

\bibitem{CGSS} P. Candela, D. Gonz\'alez-S\'anchez, B. Szegedy \emph{On nilspace systems and their morphisms}, Ergodic Theory Dynam. Systems \textbf{40} (2020), no. 11, 3015--3029.

\bibitem{CGSS-p-hom} P. Candela, D. Gonz\'alez-S\'anchez, B. Szegedy, \emph{On higher-order Fourier analysis in characteristic $p$}, Ergodic Theory Dynam.\ Systems, First View, 1--70.

\bibitem{CGSS-aff-ex} P. Candela, D. Gonz\'alez-S\'anchez,  B. Szegedy \emph{On $\mb{F}_2^\omega$-affine exchangeable probability measures} Preprint 2022. \url{https://arxiv.org/abs/2203.08915}

\bibitem{CScouplings} P. Candela, B. Szegedy, \emph{Nilspace factors for general uniformity seminorms, cubic exchangeability and limits}, Mem. Amer. Math. Soc. \textbf{287} (2023), no. 1425, v+101 pp.

\bibitem{Nikos} N. Frantzikinakis, \emph{Some open problems on multiple ergodic averages}, Bull. Hellenic Math. Soc. \textbf{60}
(2016), 41--90.

\bibitem{Furstenberg} H. Furstenberg, \emph{Recurrence in ergodic theory and combinatorial number theory}. M. B. Porter Lectures 
Princeton University Press, Princeton, N.J., 1981. xi+203 pp.

\bibitem{GSz} W. T. Gowers, \emph{A new proof of Szemer\'edi's theorem},  Geom. Funct. Anal. \textbf{11} (2001), no. 3, 465--588. 

\bibitem{Glasner} E. Glasner, \emph{A family of distal functions and multipliers for strict ergodicity},  Topological Methods in Nonlinear Analysis. Online. 23 June 2023. pp. 1--20. 

\bibitem{GGY} E. Glasner, Y. Gutman, X. Ye, \emph{Higher order regionally proximal equivalence relations for general minimal group actions}, Advances in Mathematics \textbf{333} (2018), 1004--1041.

\bibitem{GMU} E. Glasner, M. Megrelishvili, V.V. Uspenskij, \emph{On metrizable enveloping semigroups}, Isr. J. Math. \textbf{164}, 317--332 (2008).

\bibitem{GL} Y. Gutman, Z. Lian, \emph{Strictly ergodic distal models and a new approach to the Host-Kra factors}, J. Funct. Anal. \textbf{284} (2023), no. 4, Paper No. 109779, 54 pp.

\bibitem{GMV3} Y. Gutman, F. Manners, P. P. Varj\'u, \emph{The structure theory of nilspaces III: Inverse limit representations and topological dynamics}, Adv. Math. \textbf{365} (2020), 107059, 53 pp.

\bibitem{Hass&Katok} B. Hasselblatt and A. Katok. Principal structures. \emph{Handbook of Dynamical Systems}. Vol. 1A. Eds. B. Hasselblatt and A. Katok. North-Holland, Amsterdam, 2002, pp. 1–203.

\bibitem{HKbook} B. Host, B. Kra, \emph{Nilpotent structures in ergodic theory}, Mathematical Surveys and Monographs, 236. American Mathematical Society, Providence, RI, 2018.

\bibitem{HK-par} B. Host, B. Kra \emph{Parallelepipeds, nilpotent groups, and Gowers norms}, Bull. Soc. Math. France \textbf{136} (2008), 405--437.

\bibitem{HK-non-conv} B. Host, B. Kra, \emph{Nonconventional ergodic averages and nilmanifolds}, Ann. of Math. (2) \textbf{161} (2005), no. 1, 397--488.

\bibitem{JST1} A. Jamneshan, O. Shalom, T. Tao, \emph{The structure of totally disconnected Host--Kra--Ziegler factors, and the inverse theorem for the $U^k$-Gowers uniformity norms on finite abelian groups of bounded torsion}, preprint. \url{https://arxiv.org/abs/2303.04860}

\bibitem{JST2} A. Jamneshan, O. Shalom, T. Tao, \emph{A Host--Kra $\mb{F}_2^\omega$-system of order 5 that is not Abramov of order 5, and non-measurability of the inverse theorem for the $U^6(\mb{F}_2^n)$ norm}, preprint. \url{https://arxiv.org/abs/2303.04853}

\bibitem{Munkres} J. R. Munkres, \emph{Topology, Second Edition}, Prentice Hall, Inc., Upper Saddle River, NJ, 2000.

\bibitem{O&V} K. Oliveira, M. Viana, \emph{Foundations of ergodic theory}, Cambridge Stud. Adv. Math., 151
Cambridge University Press, Cambridge, 2016. xvi+530 pp.

\bibitem{Royden-Fitzpatrick} H. L. Royden, P. M. Fitzpatrick, \emph{Real analysis}, Pearson, 4th edition (January 15, 2010)

\bibitem{Salehi} E. Salehi, \emph{Distal functions and unique ergodicity}, Trans. Amer. Math. Soc. \textbf{323} (1991), no.2, 703--713.

\bibitem{Shalom} O. Shalom \emph{Multiple ergodic averages in abelian groups and Khintchine type recurrence}, Trans. Amer. Math. Soc. \textbf{375} (2022), 2729--2761.

\bibitem{Shalom2} O. Shalom, \emph{Host-Kra theory for $\bigoplus_{p\in P}\mb{F}_p$-systems and multiple recurrence}, Ergodic Theory Dynam. Systems \textbf{43} (2023), no. 1, 299--360. 

\bibitem{Shalom3} O. Shalom, \emph{Host-Kra factors for $\bigoplus_{p\in P}\mb{Z}/p\mb{Z}$ actions and finite dimensional nilpotent systems}, to appear in Anal.\ PDE. \url{https://arxiv.org/abs/2105.00446}

\bibitem{Song-Ye} S. Shao, X. Ye, \emph{Regionally proximal relation of order d is an equivalence one for minimal systems and a combinatorial consequence},
Adv. Math. \textbf{231} (2012), no. 3-4, 1786--1817. 

\bibitem{TZ-High} T. Tao, T. Ziegler, \emph{The inverse conjecture for the Gowers norm over finite fields via the correspondence principle}, Anal. PDE \textbf{3} (2010), 1--20.

\bibitem{TZ-Low} T. Tao, T. Ziegler, \emph{The inverse conjecture for the Gowers norms over finite fields in low characteristic}, Ann. Comb. \textbf{16} (2012), 121--188.

\bibitem{Wolfsurvey} J. Wolf, \emph{Finite field models in arithmetic combinatorics -- ten years on}, 
Finite Fields Appl. \textbf{32} (2015), 233--274. 

\bibitem{Ziegler} T. Ziegler, \emph{Universal characteristic factors and Furstenberg averages}, J. Amer. Math. Soc. \textbf{20} (2007), no. 1, 53--97.

\end{thebibliography}
\end{document}